\documentclass[12pt, reqno]{amsart}
\usepackage{amsmath, amsthm, amscd, amsfonts, amssymb, graphicx, color}
\usepackage[bookmarksnumbered, colorlinks, plainpages]{hyperref}

\usepackage{cleveref}

\usepackage{enumerate}
\hypersetup{colorlinks=true,linkcolor=red, anchorcolor=green, citecolor=cyan, urlcolor=red, filecolor=magenta, pdftoolbar=true}

\newtheorem{theorem}{Theorem}[section]

\newtheorem{proposition}[theorem]{Proposition}
\newtheorem{corollary}[theorem]{Corollary}
\theoremstyle{definition}
\newtheorem{definition}[theorem]{Definition}
\newtheorem{example}[theorem]{Example}
\newtheorem{remark}[theorem]{Remark}

\newtheorem{problem}[theorem]{Problem}
\theoremstyle{remark}

\numberwithin{equation}{section} 

\begin{document}

\setcounter{page}{1}

\author{Arunkumar  C.S., Sruthymurali}

\address{Arunkumar C.S., Government College Kasaragod, Vidyanagar , Kasaragod 671123, India.}
\email{\textcolor[rgb]{0.00,0.00,0.84}{arunkumarcsmaths9@gmail.com}} 

\address{Sruthymurali, Department of Mathematical Sciences, Kannur University, Kannur 670567, India.}
\email{\textcolor[rgb]{0.00,0.00,0.84}{sruthy92smk@gmail.com, sruthymurali@kannuruniv.ac.in} } 



\title{Noncommutative BKW-operators} 
\begin{abstract}
    Inspired by the classical Bohman-Korovkin-Wulbert (BKW) operators, we initiate a study of noncommutative BKW-operators. Let $A$ be a unital $C^*$-algebra, and $S$ be a set of generators of $A$. A unital completely positive (UCP)-map $\phi: A\rightarrow B(H)$ is said to be  a \textit{noncommutative BKW-operator} for $S$ with respect to norm or weak operator topology (WOT) or strong operator topology (SOT) if for any sequence of UCP-maps
$\phi_n:A\rightarrow B(H)$, $n=1,2,...,$
$\lim_{n\rightarrow \infty}\phi_n(s)=\phi(s),\forall ~s\in S$ in norm (or  WOT or SOT) $\Rightarrow \lim_{n\rightarrow \infty}\phi_n(a)=\phi(a), \forall ~a\in A$ in norm (or WOT or SOT, respectively).  We identify a connection between noncommutative BKW-operators and the unique CP-extension of UCP-maps. We have discussed several examples and explored different notions of noncommutative BKW-operators and their interconnections. Additionally, we introduce the concept of hyperrigidity with respect to a UCP-map and characterize it along the lines of Arveson. Although independent yet related to noncommutative BKW-operators, we provide a noncommutative version  of operator version of the Korovkin theorem recently proposed by D. Popa.

\end{abstract}
\maketitle

\section{Introduction} 

In 1952, H. Bohman \cite{Boh52} established a significant result by demonstrating the strong convergence of a sequence of special interpolation operators $\{B_{n}:n\in \mathbb{N}\}$ on $C[0,1]$ towards the identity operator I on $C[0,1]$. Bohman's theorem states that the sequence $B_{n}$ converges to $I$ strongly whenever $\{B_{n}(t^m):n\in \mathbb{N}\}$ uniformly converges to the functions $t^m$ for each $m=0,1,2$. In 1953, P.P. Korovkin \cite{Kor60} proved a remarkable result that Bohman's theorem is  still holds even when the interpolation operators are replaced by positive linear operators on $C[0,1]$.  Later,  in 1968, D.E. Wulbert \cite{Wul68} demonstrated the validity of Korovkin's theorem even when the positivity condition on the operators is relaxed. Instead, the convergence criterion is based on the operator norms, especially focusing on sequences where the operator norms converge to one.

Let $X$ denote a compact Hausdorff space, and consider a set $G\subset C(X)$. If for any sequence $\psi_n:C(X)\rightarrow C(X)$ of positive linear maps where $\lim\limits_{n\rightarrow\infty} ||\psi_{n}(g)-g ||=0 $ holds for every $g\in G$, then we also have $\lim\limits_{n\rightarrow\infty} ||\psi_{n}(f)-f ||=0 $ for every $f\in C(X)$, then $G$ is referred to as a Korovkin set in $C(X)$. The Korovkin theorem can be restated as follows:  The set $G=\{1,x,x^2\}$ is a Korovkin set for  $C(X)$. A natural generalization of Korovkin sets in $C(X)$ to Banach spaces was proposed by Berens and Lorentz \cite{BL75} as follows: Let $E$ and $F$ be two Banach spaces. Let $\mathcal{B}$ be a subset of the set of all bounded linear operators from $E$ to $F$, and let $P$ be a fixed operator of the class $\mathcal{B}$. A set $G\subseteq E$ is called a Korovkin set with respect to the pair $(\mathcal{B}, P)$, if for each sequence $\psi_{n}\in \mathcal{B}$ with $\Vert \psi_{n} \Vert \rightarrow 1$, the relation $\psi_{n}(x)\rightarrow Px$, $x\in G$ (in norm) implies $\psi_{n}(x)\rightarrow Px$, $x\in E$. In this case the operator $P$ is called Bohman-Korovkin-Wulbert operator(in short BKW-operator) for the set $G$ \cite{Tak90}. One can easily see that a set $G$ qualifies as a Korovkin set for $(\mathcal{B}, P)$ (or alternatively, $P$ serves as a BKW-operator for $G$) if and only if the subspace $S$ generated by $G$ is also a Korovkin set for $(\mathcal{B}, P)$ (or $P$ is a BKW-operator for $S$). Furthermore, it is worth mentioning that any dense subset $G$ of $E$ trivially qualifies as a Korovkin set for $(\mathcal{B}, P)$, implying that all elements of $\mathcal{B}$ are BKW-operators for any dense subset $G$ of $E$.  Note that, in the case of the usual Korovkin set, $\mathcal{B}$ is the set of all positive operators and $P=I$. Despite generalizing the concept of Korovkin sets in $C(X)$ to arbitrary Banach spaces, Berens and Lorentz focused their study on the case where $E=F=C(X)$ and $P=I$.

The generalized definition of Korovkin sets inherently gives rise to two fundamental questions. The first question is focused on identifying conditions under which a nontrivial Korovkin subset exists for a given $P\in \mathcal{B}$. The objective is to determine some or all of these conditions. The second question pertains to a given subspace $S$ of $E$, aiming to determine some or all of the  operators  $P\in \mathcal{B}$ (if they exist) for which $S$ serves as a Korovkin set for $(\mathcal{B}, P)$. Some aspects of the first question were discussed in \cite[Section 3.3 and 3.4]{AC94} and in \cite[Section 5]{Alt10}. 
%%%%%%%%%%%%%%%%%%%%%%%%%%%%%%%%%%%%%%%%%%%%%%
Regarding the second question, several results are available in the literature. BKW-operators have been studied on $C[0,1]$ \cite{IT97}, and BKW-operators on the disk algebra for the test functions $\{1, z\}$ have been determined \cite{ITW96b}. Weighted composition operators that qualify as BKW-operators have been characterized \cite{ITW96a}, and BKW-operators for the Chebyshev system have been investigated \cite{II99}. S.-E. Takahasi, in a series of papers \cite{Tak90,Tak93,Tak95,Tak96}, explored BKW-operators in various settings, including $C[0,1]$ for the test functions $\{1, x, x^2, x^3, x^4\}$, normed spaces, characterizations in terms of Korovkin closures, and mappings from function spaces to commutative $C^*$-algebras. See also \cite[Section 10]{Alt10}. 
%%%%%%%%%%%%%%%%%%%%%%%%%%%%%%%%%%%%%%%

The noncommutative approximation  theory introduced by Arveson \cite{Arv11} in the realm of operator systems within $C^*$-algebras have experienced significant advancements in recent times. 
The main purpose of this article is to  investigate the noncommutative analogues of generalised Korovkin sets  and BKW-operators in the setting of $C^*$-algebras and unital completely positive maps. Let $A$ and $B$ be unital $C^*$-algebras. A linear map $\phi:A\rightarrow B$ induces a linear map $\phi^{(n)}:M_{n}(A)\rightarrow M_{n}(B)$, via
$$\phi^{(n)}([a_{ij}]_{n\times n})=[\phi(a_{ij})]_{n\times n}.$$ A linear map $\phi:A\rightarrow B$ is called positive if $\phi(a)\geq 0$ in $B$ whenever $a\geq 0$ in $A$. The map $\phi$ is called completely positive(CP) if $\phi^{(n)}$ is positive for all $n\in \mathbb{N}$. A CP-map $\phi$ is called unital completely positive(UCP) if $\phi(1)=1$.  A subspace within a $C^*$-algebra that is both self-adjoint and possesses a unit element is referred to as an operator system. UCP-maps on operator systems are defined in a manner similar to the definition of UCP-maps on $C^*$-algebras. Let $B(H)$ denote the $C^*$-algebra of all bounded operators on a Hilbert space $H$.  Let us denote $UCP(S,B(H))$ by the collection of all UCP-map from $S$ to $B(H)$. The set $UCP(S,B(H))$ is compact with respect to the BW-topology, where  a net $\phi_{\lambda}$ converges to $\phi$ in the BW-topology if $\phi_{\lambda}(a)\rightarrow \phi(a)$ in the weak operator topology of $B(H)$. We refer to \cite{Pau02} for a comprehensive and well-explained theory of operator systems and CP-maps.

 Let $S$ be an operator system and let $C^*(S)$ be the $C^*$-algebra generated by $S$. In Section \ref{Noncommutative BKW-operators}, we show a completely positive map $\phi:C^*(S)\rightarrow B(H)$ is a noncommutative BKW-operator for the operator system $S$ if and only if $\phi |_{S}$ has unique CP-extension to the $C^*$-algebra $C^*(S)$. Also, the existence of noncommutative BKW-operator  for any given operator system is addressed. Additionally, we have discussed several examples and explored different notions of noncommutative BKW-operators and their interconnections. In Section \ref{hyperrigidity}, we introduce the concept of hyperrigidity with respect to a UCP-map and characterize it along the lines of Arveson. In the final section, although independent yet related to noncommutative BKW-operators, we provide a noncommutative version  of operator version of the Korovkin theorem recently proposed by D. Popa.\\
 
We organize the main results of this paper as follows. The first result provides a necessary and sufficient condition for a UCP-map to be a noncommutative BKW-operator.

\bigskip
\noindent\textbf{Theorem A:} (See  \Cref{BKW-operators and unique CP-extensions})
Let $S$ be an operator system in a $C^*$-algebra $A=C^*(S)$. Then for a UCP-map $\phi: A \rightarrow B(H)$, the following conditions are equivalent:
\begin{enumerate}[(i)]   
    \item $\phi$ is a noncommutative BKW-operator for $S$ with respect to WOT. 
    \item $\phi|_{S}$ has unique CP-extension to $A$. 
\end{enumerate}
\bigskip
Motivated by the well known notion of hyperrigidity introduced by Arveson in \cite{Arv11}, we introduce and study hyperrigidity with respect to a UCP-map which generalizes Arveson's result.  Regarding Arverson-type results, it is worth mentioning the work of Limaye and Namboodiri \cite{LN84b}, and later works of Namboodiri and et al(see \cite{MNN11,MNN12,NPSV18} and
references therein). Furthermore, a fruitful connection between the Korovkin theory and fast
methods for large linear systems can be found in \cite{FFHMS19,Ser99,KNS13}. See Item B) at page 1076 in \cite{FFHMS19} and the
second part of Theorem 3.6 in \cite{Ser99}, while noncommutative Korovkin theorems can be found in \cite{KNS13}.

We provide a characterization of hyperrigidity with respect to a UCP as follows:\\

\noindent\textbf{Theorem B:} (See Theorem \ref{characterisation})
Let $S$ be a separable operator system that generate  a $C^*$-algebra $A$. Let $\eta: A\rightarrow A$ be a UCP-map. Then the following statements are equivalent: 
 \begin{enumerate}
     \item[(i)] $S$ is  $\eta$-hyperrigid in $A$. 
     \item[(ii)] For every non-degenerate representation $\pi: A \rightarrow B(H)$ of $A$ on  a separable Hilbert space $H$  and for any sequence of UCP-maps $\phi_{n}: A \rightarrow B(H)$ such that  $$\lim\limits_{n\rightarrow \infty}\Vert \phi_{n}(s)-\pi(\eta(s))\Vert =0 ~\forall s\in S \implies $$ 
     $$ \lim\limits_{n\rightarrow \infty}\Vert \phi_{n}(a)-\pi(\eta(a))\Vert =0~\forall a\in A.$$
     \item[(iii)] For every non-degenerate representation $\pi: A \rightarrow B(H)$ of $A$ on a separable Hilbert space $H$, the UCP-map $\pi\circ \eta _{|S}$ has unique CP-extension to $A$, namely $\pi \circ \eta $  itself.
    
     \item[(iv)] For any $C^*$-algebra $B$, and for any  unital $*$-homomorphism $\rho: A \rightarrow B$ and UCP-map $\phi: B \rightarrow B$ such that $\phi(\rho(s))=\rho(\eta(s)) ~\forall s \in S$ then $\phi(\rho(a))=\rho(\eta(a)) ~\forall a \in A$.
 \end{enumerate}
\bigskip
In the next result we provide a noncommutative version of noncommutative analogue of the operator version of Korovkin’s theorem, recently established by Popa\cite[Theorem 1]{Pop22}. We prove the result for more general maps—namely, Schwarz maps and 2-positive maps and deduce the corresponding result for UCP maps as a corollary. Let $A$ and $B$ be unital $C^*$-algebras. Fix an element $a \in A$, and denote by
$A_0 = C^*(a)$, the $C^*$-subalgebra of A generated by a.  \\

\noindent\textbf{Theorem C:} (See Theorem \ref{noncommutative Popa's theorem} and Corollary \ref{noncommutative  theorem})\\
\noindent
$\bullet~~~$ Let  $\phi_{n}: A_{0}\rightarrow B$ be a sequence of Schwarz maps and $\phi: A\rightarrow B$ be a 2-positive linear map  such that $\phi(1)$ is invertible,   $\phi(a^*a)=\phi(a^*)\phi(1)^{-1}\phi(a)$ and  $\phi(aa^*)=\phi(a)\phi(1)^{-1}\phi(a^*)$. If $\lim\limits_{n\rightarrow \infty} \phi_{n}(1)=\phi(1)$,  $\lim\limits_{n\rightarrow \infty} \phi_{n}(a)=\phi(a)$,
$\lim\limits_{n\rightarrow \infty} \phi_{n}(a^*a)=\phi(a^*a)$, $\lim\limits_{n\rightarrow \infty} \phi_{n}(aa^*)=\phi(aa^*)$ in norm, then for every $x\in C^*(a)$, $\lim\limits_{n\rightarrow \infty} \phi_{n}(x)=\phi(x)$ in norm. \\\\
$\bullet$ Let  $\phi,\phi_{n}: A_{0}\rightarrow B$ be a sequence of CP-maps  such that $\phi(1)$ is invertible,   $\phi(a^*a)=\phi(a^*)\phi(1)^{-1}\phi(a)$ and  $\phi(aa^*)=\phi(a)\phi(1)^{-1}\phi(a^*)$. If $\lim\limits_{n\rightarrow \infty} \phi_{n}(1)=\phi(1)$,  $\lim\limits_{n\rightarrow \infty} \phi_{n}(a)=\phi(a)$,
$\lim\limits_{n\rightarrow \infty} \phi_{n}(a^*a)=\phi(a^*a)$, $\lim\limits_{n\rightarrow \infty} \phi_{n}(aa^*)=\phi(aa^*)$ in norm, then for every $x\in C^*(a)$, $\lim\limits_{n\rightarrow \infty} \phi_{n}(x)=\phi(x)$ in norm.

\section{Noncommutative BKW-operators}\label{Noncommutative BKW-operators}

The main goal of this section is to introduce the concept of noncommutative BKW-operators, a necessary and sufficient condition for a UCP map to be a noncommutative BKW-operator and provide several illustrative examples.

\begin{definition}\label{noncommutative BKW-operators}
Let $A$ be a $C^*$-algebra  and $S$ be a set of generators of $A$. A UCP-map $\phi: A\rightarrow B(H)$ is said to be  a \textit{noncommutative BKW-operator} for $S$ with respect to norm topology if for any sequence of UCP maps 
$\phi_n:A\rightarrow B(H)$, $n=1,2,...,$
$$\lim_{n\rightarrow \infty}||\phi_n(s)-\phi(s)||=0,\forall ~s\in S\Rightarrow \lim_{n\rightarrow \infty}||\phi_n(a)-\phi(a)||=0, \forall ~a\in A.$$ 
\end{definition} 

In the context of UCP-maps, the norm of $\phi$ is equal to 1, which implies that the condition $\Vert \phi_{n}\Vert \rightarrow 1$ is satisfied trivially. If we consider the set $\mathcal{B}$ to consist of all UCP-maps from $A$ to $B(H)$ in the classical definition of BKW-operators mentioned in the introduction, we obtain the definition of noncommutative BKW-operators.

\begin{remark}
Let $\phi:A\rightarrow B(H)$ be a noncommutative BKW-operator for $S\subseteq A$. Then $U^*\phi U$ is also a noncommutative BKW-operator for $S$, where $U$ is any unitary operator on $H$. This is an immediate consequence of the fact that the operator norm is unitarily invariant.
\end{remark}
\begin{example}
   Consider $A=C[0,1]$ and $S=\{1,t,t^2 \}\subseteq A$. For $f\in A$, let $M_{f}$ be the multiplication operator on $L^2[0,1]$. Consider the completely positive map $\phi : A\rightarrow B(L^2[0,1])$, defined as $\phi(f)=M_{f}$. Then $\phi$ is a noncommutative BKW-operator for $S$ is a consequence of the facts that $\phi$ is a $*$-homomorphism and completely positive maps are Schwarz maps. %see for instance \cite[Corollary 1.2]{LN84}. 
   This specific scenario is a particular instance of the following more general example. Let $A$ be any $C^*$-algebra and element $a\in A$. Consider $S=\{1,a,a^*a,aa^* \}$. Then any representation $\pi:A\rightarrow B(H)$ is a noncommutative BKW-operator for $S$, see \cite[Corollary 1.2]{LN84} for more details.
\end{example}

\begin{example}
According to \cite[Theorem 3.2]{Arv11}, if we have a self-adjoint operator $T\in B(H)$ and $A$ represents the $C^*$-algebra generated by $T$, then the inclusion map $\phi$ from $A$ to $B(H)$ serves as a noncommutative BKW-operator for the set $S=\{1,T,T^2\}$. However,  if the spectrum of $T$ contains at least three distinct points, then the map $\phi$ does not qualify as a noncommutative BKW-operator for the set $\{1,T \}$. \color{black}
\end{example}

Next we provide a non example of a noncommutative BKW-operator.

\begin{example}
Consider a separable infinite-dimensional Hilbert space $H$ and the unilateral right shift operator $V$ on $H$. Let $\phi: C^*(V)\rightarrow B(H)$ be defined as $$\phi(X) = V^*XV \text{ for all }X\in C^*(V).$$ Let $S$ be the operator  system generated by $V$, that is, $$S=\{\lambda_{1}I+\lambda_{2}V+\lambda_{3}V^*: \lambda_{1}, \lambda_{2}, \lambda_{3}\in \mathbb{C}\}.$$
Define $\phi_{n}:C^*(V)\rightarrow B(H)$ by $\phi_{n}=Id_{C^*(V)}$ $\forall~ n\in \mathbb{N}$. Then, for all elements 
$ \lambda_{1}I+\lambda_{2}V+\lambda_{3}V^*\in S$,
\begin{align*}
\lim_{n\rightarrow \infty}||\phi_n(\lambda_{1}I+\lambda_{2}V+\lambda_{3}V^*)-\phi(\lambda_{1}I+\lambda_{2}V+\lambda_{3}V^*)|| &= \\
||\lambda_{1}I+\lambda_{2}V+\lambda_{3}V^*-V^*(\lambda_{1}I+\lambda_{2}V+\lambda_{3}V^*)V||=0. 
\end{align*}
But, for $VV^*\in C^*(V)$, 
$$\lim_{n\rightarrow \infty}||\phi_n(VV^*)-\phi(VV^*)||= 
||VV^*-I||\neq 0. $$ 
Thus, the map $\phi$ is not a noncommutative BKW-operator for $S$.
\end{example}

It is clear that the Definition \ref{noncommutative BKW-operators} of a noncommutative BKW-operator make use of the norm convergence in $B(H)$. However, an alternative definition of noncommutative BKW-operator can be provided where the norm convergence is replaced with weak convergence, specifically the BW-convergence of UCP-maps.

\begin{definition}\label{weak noncommutative BKW-operators}
Let $A$ be a $C^*$-algebra  and $S$ be a set of generators of $A$. A UCP-map $\phi: A\rightarrow B(H)$ is said to be  a \textit{noncommutative BKW-operator} for $S$ with respect to WOT if for any sequence of UCP maps 
$\phi_n:A\rightarrow B(H)$, $n=1,2,...,$ and $\forall ~x,y\in H$,
$$\lim_{n\rightarrow \infty}\langle (\phi_n(s)-\phi(s))x,y \rangle=0,\forall ~s\in S\Rightarrow \lim_{n\rightarrow \infty}\langle (\phi_n(a)-\phi(a))x,y \rangle=0, \forall ~a\in A.$$ 
\end{definition} 

 Next, we provide a desirable property of CP-map that qualifies it as a noncommutative BKW-operator in the sense of Definition \ref{weak noncommutative BKW-operators}. 

\begin{theorem}\label{BKW-operators and unique CP-extensions}
Let $S$ be a separable operator system and $A=C^*(S)$. Then for a  UCP-map $\phi: A \rightarrow B(H)$, where $H$ is a separable Hilbert space, the following conditions are equivalent:
\begin{enumerate}[(i)]   
    \item $\phi$ is a noncommutative BKW-operator for $S$ with respect to WOT. 
    \item $\phi|_{S}$ has unique CP-extension to $A$. 
\end{enumerate}
\end{theorem}

\begin{proof}
$(i)\implies (ii)$: Let $\psi : A\rightarrow B(H)$ be a UCP extension of  $\phi|_{S}$. Take $\phi_{n}=\psi$ for all $n\in \mathbb{N}$. Then $\phi_{n}(s)$  converges to $\phi(s)$ in the weak operator topology for every $s\in S$. By assumption $(i)$ we have  $\phi_{n}(a)$  converges to $\phi(a)$ in the weak operator topology for every $a\in A$. Thus $\psi(a)=\phi(a)$ for every $a\in A$. \\
$(ii)\implies (i)$: Let $\phi_{n}:A\rightarrow B(H)$ be a sequence of UCP maps such that 
$$\lim \limits_{n\rightarrow \infty}\langle ( \phi_{n}(s)-\phi(s)) x,y \rangle =0 ~~\forall ~ s\in S,~ x,y \in H.$$
We have to show that $$\lim \limits_{n\rightarrow \infty}\langle ( \phi_{n}(a)-\phi(a)) x,y \rangle =0 ~~\forall ~ a\in A,~ x,y \in H.$$
In other words, we need to prove that  the sequence  $\{\phi_{n}\}$ converges to $\phi$ in the BW-topology of $UCP(A,B(H))$. Since $UCP(A, B(H))$ with  BW-topology  is metrizable, it is enough to show every subsequence $\{\phi_{n_{k}}\}$ of  
$\{\phi_{n}\}$ contains a convergent subsequence. Let $\{\phi_{n_{k}}\}$ be a subsequence of   $\{\phi_{n}\}$. Since $UCP(A, B(H))$ is compact in the BW-topology  there is a subsequence $\{\phi_{n_{k,l}}\}$ of $\{\phi_{n_{k}}\}$ that converges in the BW-topology, say to $\psi\in UCP(A, B(H))$. As for $s\in S$, $\phi_{n}(s)\rightarrow \phi(s)$ in the weak operator topology we must have  $\psi(s)=\phi(s)$ for all $s\in S$. By assumption $(ii)$, it follows that $\psi =\phi$ on $A$. 
\end{proof}

The \Cref{BKW-operators and unique CP-extensions}  can be employed to generate examples for noncommutative BKW-operators as illustrated below.

\begin{example}
Let $\{E_{ij}:i,j=1,2\}$ denotes the standard matrix units in $M_2(\mathbb{C})$. Let $T=\begin{bmatrix} 1 & i\\ -i & 1 \end{bmatrix}$ and let  $S$ be the operator system in $M_{2}(\mathbb{C})$ generated by $\{I, E_{11}, T\}$.  Consider the UCP-map $\phi:M_{2}(\mathbb{C})\rightarrow M_{2}(\mathbb{C})$ defined by 
$$\phi \left(\begin{bmatrix} a & b\\ c& d \end{bmatrix} \right ) = \begin{bmatrix} a& 0\\ 0& d \end{bmatrix}.$$ 
Let  $\psi: M_{2}(\mathbb{C})\rightarrow M_{2}(\mathbb{C})$ be a UCP such that $\phi |_{S}=\psi |_{S}$.  Then $\psi(E_{11})=\phi(E_{11})=E_{11}$ and 
  $\psi(E_{22})=\psi(I-E_{11})=I-E_{11}=E_{22}$.
Now, we will show that $\psi(E_{12})$ is self adjoint.  Since $\psi(E_{12})^*=\psi(E_{12}^*)=\psi(E_{21})$, its enough to show that $\psi(E_{12})=\psi(E_{21})$.
Note that 
$$\psi\left (\begin{bmatrix} 0 & i\\ -i& 0 \end{bmatrix} \right )=\psi(T-I)=\phi(T-I) =0.$$
Thus $\psi(iE_{12}-iE_{21})=0$ and hence $\psi(E_{12})=\psi(E_{21})$.

Let $\psi(E_{12})$ is the form $\begin{bmatrix} a & b\\ \bar{b}& d \end{bmatrix}$ with $a,d\in \mathbb{R}$. We show that $b=d=0$.
\begin{align*}
    \begin{bmatrix} 1 & 0\\ 0& 0 \end{bmatrix} &=\psi(E_{11})=\psi(E_{12}E_{21}) \\
                                              &\geq \psi(E_{12})\psi(E_{21}) =\begin{bmatrix} a & b\\ \overline{b}& d \end{bmatrix}
                                              \begin{bmatrix} a & b\\ \overline{b}& d \end{bmatrix} \\
                                              &= \begin{bmatrix} a^2 + |b|^2 & ab+bd\\ a\overline{b}+\overline{b}d&  |b|^2+d^2 \end{bmatrix} \\
                                              & \geq 0.
\end{align*}
Form the above inequality, we can conclude that $|a|\leq1$ and $b=d=0$.
Thus $\psi(E_{12})=\begin{bmatrix} a & 0\\ 0& 0 \end{bmatrix}$ with $|a|\leq 1$.
Now, 
\begin{align*}
    \begin{bmatrix} 0 & 0\\ 0& 1 \end{bmatrix} &=\psi(E_{22})=\psi(E_{21}E_{12}) \\
                                              &\geq \psi(E_{21})\psi(E_{11}) =\begin{bmatrix} a & 0\\ 0& 0 \end{bmatrix}
                                              \begin{bmatrix} a & 0\\ 0& 0 \end{bmatrix} \\
                                              &= \begin{bmatrix} a^2 & 0\\ 0& 0 \end{bmatrix}\geq 0. 
\end{align*}
Thus $a=0$. Hence $\psi(E_{12})=0=\phi(E_{12})$. Also, we have $\psi(E_{21})=\psi(E_{12})=0=\phi(E_{12})$.
Hence $\psi=\phi$ on $M_{2}(\mathbb{C})$. That is, $\phi_{|S}$ has unique CP-extension to $M_{2}(\mathbb{C})$. Hence $\phi$ is a noncommutative BKW-operator for the operator system $S$ in the sense of Definition \ref{noncommutative BKW-operators} and Definition \ref{weak noncommutative BKW-operators}.
\end{example}

\color{black}
\begin{example}
Let $K(H)$ be the $C^*$-algebra of compact operators on a separable Hilbert space $H$, and let $T(H)$ denote the trace class operators on $H$. Let $V$ be the right shift operator on $H$. Consider the UCP-map $\phi: B(H)\rightarrow B(H)$ defined by $$\phi(x)=V^*xV.$$ Then $\phi$ is unital as $V$ is an isometry.  Consider the operator system $\widetilde{T(H)}=T(H)+\mathbb{C}I$ and the generated $C^*$-algebra $\widetilde{K(H)}=K(H)+\mathbb{C}I.$ 
Let $\widetilde{\phi}=\phi|_{\widetilde{K(H)}}$, the restriction of $\phi$ to $\widetilde{K(H)}$. We claim that $\widetilde{\phi}$ is a noncommutative BKW-operator for the operator system $\widetilde{T(H)}$, by showing that $\widetilde{\phi}|_{\widetilde{T(H)}}$ has unique CP-extension namely $\widetilde{\phi}$ itself. Let $\tilde{\psi}:\widetilde{K(H)}\rightarrow B(H)$ be a UCP-map such that $\tilde{\psi} |_{(\widetilde{T(H)}}=\tilde{\phi} |_{\widetilde{T(H)}}$. 
In particular, 
$$\tilde{\psi} |_{(T(H)}=\tilde{\phi} |_{T(H)}.$$
For $h\in H$, the rank one operator $h\otimes h \in T(H)$. Thus $\tilde{\psi}(h\otimes h)=\tilde{\phi}(h\otimes h)$. 
 \begin{align*}
     \tilde{\psi}(h\otimes h )=& \tilde{\phi}(h\otimes h)\\
    =& V^*(h\otimes h)V \\
    =& (V^*h)\otimes (V^*h).
 \end{align*}
Therefore $\tilde{\psi}(h\otimes h)$ and $\tilde{\phi}(h\otimes h)$ are both rank one operators and hence they are compact. Also, by the positivity of $\phi$ and $\psi$, and rank one  operator $h\otimes h$ we have $\tilde{\psi}(h\otimes h)=\tilde{\phi}(h\otimes h)\geq 0$. That is, $\tilde{\psi}(h\otimes h)\in K(H)^+$  and $\tilde{\phi}(h\otimes h)\in K(H)^+$ for every $h\in H$.
Applying the result \cite[Theorem 2.3]{XY16} to $\tilde{\psi}|_{K(H)}$ and $\tilde{\phi}|_{K(H)}$ we have $Range(\tilde{\psi}|_{K(H)})\subseteq K(H)$ and $Range(\tilde{\phi}|_{K(H)})\subseteq K(H)$. i.e. 
$$\tilde{\psi}|_{K(H)}:K(H)\rightarrow K(H) \text{ and }$$
$$\tilde{\phi}|_{K(H)}:K(H)\rightarrow K(H).$$
Since $\tilde{\psi} |_{(T(H)}=\tilde{\phi} |_{T(H)}$,
by \cite[Lemma 3.2(b)]{XY16} it follows that $\tilde{\psi} =\tilde{\phi}$ on {K(H)}.
Hence $\tilde{\psi}=\tilde{\phi}$ on $\widetilde{K(H)}$.
Hence $\phi$ is a noncommutative BKW-operator for the operator system $\widetilde{T(H)}$.
\end{example} 

The following simple observation can also be utilized for generating additional examples.
\begin{proposition}
Let $\psi: B\rightarrow B(H) $ be a UCP map and let $\phi: A\rightarrow B$ be a surjective  UCP map such that $B=C^*(\phi(S))$. If $\psi \circ \phi | _{S}$ has unique CP-extension to $A$, then $\psi |_{\phi(S)}$ has unique CP-extension to $B$. 
\end{proposition}
\begin{proof}
Assume that $\psi \circ \phi | _{S}$ has unique CP-extension to $A$. Let $\tilde{\psi}$ be a CP-extension of $\psi |_{\phi(S)}$ to $B$. Then
$$\tilde{\psi}|_{\phi(S)}=\psi |_{\phi(S)} \implies  \tilde{\psi}(\phi(s))=\psi(\phi(s))~\forall s\in S.$$ But the map $\psi \circ \phi | _{S}$ has unique CP-extension to $A$ gives that  $$\tilde{\psi}(\phi (a))=\psi(\phi(a))~\forall a\in A.$$ 
Since the map $\phi$ is surjective, it follows that $\tilde{\psi}(b)=\psi(b)~\forall ~b\in B.$
\end{proof}

We remark that, in general a UCP-map on an operator systems can have more than one UCP-extensions to the generated $C^*$-algebra which we illustrate with an example.

\begin{example}
Let $S$ be the operator system in $M_{2}(\mathbb{C})$ given by $$ S=\left \{ \begin{bmatrix} a & b\\ c& a \end{bmatrix} : a,b,c \in \mathbb{C} \right \}.$$
Consider the UCP-map $\phi:M_{2}(\mathbb{C})\rightarrow M_{2}(\mathbb{C})$ defined by 
$$\phi \left(\begin{bmatrix} a & b\\ c& d \end{bmatrix} \right ) = \begin{bmatrix} a& 0\\ 0& d \end{bmatrix}.$$ 
Then the map $\psi: M_{2}(\mathbb{C})\rightarrow M_{2}(\mathbb{C})$ given by 
$$\psi \left(\begin{bmatrix} a & b\\ c& d \end{bmatrix} \right ) = \begin{bmatrix} d& 0\\ 0& a \end{bmatrix}$$ is  a UCP extension of  $\phi |_{S}$ but $\phi \neq \psi$. Thus  $\phi |_{S}$ has more than one CP-extension to $C^*(S)=M_{2}(\mathbb{C})$.
\end{example}

In light of \Cref{BKW-operators and unique CP-extensions}, our investigation of noncommutative BKW-operators in the sense of Definition \ref{weak noncommutative BKW-operators} leads us to examine UCP-maps on operator systems  with unique  CP-extension to the generated C*-algebras. Based on the well-known  Arveson extension theorem, which states that: suppose  $\psi: S\rightarrow B(H)$ is a CP-map then there exists a CP-map $\tilde{\psi}: A\rightarrow B(H)$ such that $\tilde{\psi}_{|S}=\psi$, we can assert the following statement. The investigation of noncommutative BKW-operators can be examined by the uniqueness of the Arveson extension of UCP-maps on operator systems.

The representations of $C^*$-algebras form one of the simplest classes of completely positive maps. 
Arveson's seminal papers\cite{Arv69,Arv72} focused on the study of irreducible representations that possess the unique extension property. Let $S$ be an operator system, and let $A=C^*(S)$. A representation $\pi:A\rightarrow B(H)$ is said to have unique extension property for  $S$ if $\pi$ is the only UCP-extension of $\pi_{|S}$ from $S$ to $A$. An irreducible representation of $A$ with unique extension property for $S$ is called a boundary representation for the operator systems $S$. Thus by \Cref{BKW-operators and unique CP-extensions}, each boundary representation for $S$ is a noncommutative BKW-operator for $S$. The existence of boundary representations, which was proven by Arveson for separable operator systems \cite{Arv08}, and for general cases by Davidson and Kennedy \cite{DK15}, ensures the existence of noncommutative BKW-operators in the sense of Definition \ref{weak noncommutative BKW-operators}. 

\begin{theorem}\label{existence of BKW-operator}
    There exists a noncommutative BKW-operator for any  operator system $S$ in a $C^*$-algebra $A$.
    \end{theorem}
\begin{proof}
    Directly follows from the existence of boundary representation \cite{DK15}.
\end{proof}

In the Definition \ref{weak noncommutative BKW-operators}, we utilized the weak operator topology (WOT) on $B(H)$. It is also feasible to define a similar concept using the strong operator topology (SOT) in $B(H)$ as shown below. 

\begin{definition}\label{strong noncommutative BKW-operators}
Let $A$ be a $C^*$-algebra  and $S$ be a set of generators of $A$. A UCP-map $\phi: A\rightarrow B(H)$ is said to be  a \textit{noncommutative BKW-operator} for $S$ if for any sequence of UCP-maps 
$\phi_n:A\rightarrow B(H)$, $n=1,2,...,$  $\forall~ h\in H$
$$\lim_{n\rightarrow \infty}||\phi_n(s)h-\phi(s)h||=0,\forall ~s\in S\Rightarrow \lim_{n\rightarrow \infty}||\phi_n(a)h-\phi(a)h||=0, \forall ~a\in A.$$ 
\end{definition} 
However, it is crucial to observe that the weak and strong operator topologies on $B(H)$ yield the same notion of noncommutative BKW-operator as shown in the following theorem.
\begin{theorem}
   A UCP-map $\phi$ is a noncommutative BKW-operator in the sense of  Definition  \ref{weak noncommutative BKW-operators} if and only if $\phi$ is a noncommutative BKW-operator in the sense of Definition \ref{strong noncommutative BKW-operators}.
\end{theorem} 
\begin{proof}
Assume that $\phi$ is a noncommutative BKW-operator in the sense of  Definition  \ref{weak noncommutative BKW-operators}. Let $\phi_{n}(s)\rightarrow \phi(s)$ in SOT,  $\forall s\in S$. Then $\phi_{n}(s)\rightarrow \phi(s)$ in WOT,  $\forall s\in S$. By assumption $\phi_{n}(a)\rightarrow \phi(a)$ in WOT, $\forall a\in A$.   Now, $\forall a\in A$ and $h\in H$,
\begin{align*}
    \Vert \phi_{n}(a)h-\phi(a)h \Vert^2 
    &= \langle \phi_{n}(a)h-\phi(a)h,\phi_{n}(a)h-\phi(a)h \rangle \\
    &=  \langle \phi_{n}(a)h,\phi_{n}(a)h \rangle +\langle \phi(a)h , \phi(a)h \rangle -2Re\langle \phi_{n}(a)h , \phi(a)h \rangle \\   
    &=  \langle \phi_{n}(a)^*\phi_{n}(a)h, ~h \rangle +\langle \phi(a)^*\phi(a)h ,~ h \rangle -2Re\langle \phi_{n}(a)h , \phi(a)h \rangle \\  
    &\leq \langle \phi_{n}(a^*a)h, h \rangle + \langle \phi(a^*a)h , h \rangle -2Re\langle \phi_{n}(a)h , \phi(a)h \rangle \\ 
    & \rightarrow   \langle \phi(a^*a)h , h \rangle+ \langle \phi(a^*a)h , h \rangle- 2Re\langle \phi(a^*a)h , h \rangle=0.
\end{align*}
Thus, $\phi$ is a noncommutative BKW-operator in the sense of Definition \ref{strong noncommutative BKW-operators}.

For the other implication, assume that $\phi$ is a noncommutative BKW-operator in the sense of  Definition  \ref{strong noncommutative BKW-operators}. By \Cref{BKW-operators and unique CP-extensions}, it's enough prove that $\phi_{|S}$ has unique UCP-extension. Suppose $\psi$ is a UCP-extension of  $\phi_{|S}$ then taking  $\phi_{n}=\psi$ on $A$, $\forall n$; we have $\phi_{n}(s)\rightarrow \phi(s)$ in SOT $\forall s\in S$. Then by assumption $\phi_{n}(a)\rightarrow \phi(a)$ in SOT $\forall a\in A$. That is, $\psi(a)=\phi(a)$  $\forall a\in A$. 
\end{proof}

\section{Hyperrigidity with respect to a UCP map}\label{hyperrigidity}
In this section we introduce and study hyperrigidity with respect to a UCP-map which generalizes the well-known notion of hyperrigidity due to Arveson. Before defining the concept of hyperrigidity with respect to a UCP-map, let us first recall the notion of hyperrigidity introduced by Arveson in \cite{Arv11}.

\begin{definition}\cite{Arv11}\label{hyperrigidity of Arveson}
Let $A$ be a $C^*$algebra and $S$ be a set of generators of $A$. Then 
$S$ is said to be  hyperrigid in $A$ if for every faithful representation of  $\sigma:A\rightarrow B(H)$ and every sequence of UCP maps 
$\psi_n:B(H)\rightarrow B(H)$, $n=1,2,...,$
$$\lim_{n\rightarrow \infty}||\psi_n(\sigma{(s)})-\sigma{(s)}||=0,\forall ~s\in S\Rightarrow \lim_{n\rightarrow \infty}||\psi_n(\sigma{(a)})-\sigma{(a)}||=0, \forall ~a\in A.$$ 
\end{definition}

We can rephrase the Definition \ref{hyperrigidity of Arveson} interms of noncommutative BKW-operators as follows. Let $\sigma: A \to B(H)$ be any faithful representation of $A$. If we identify $\sigma(S)$ with $S$, then the Definition \ref{hyperrigidity of Arveson} implies that  identity map is a noncommutative BKW-operator for $S$. Motivated by this, we introduce hyperrigidity with respect to a UCP-map.
%In the above definition, suppose  S is an operator system. If we identifying $\sigma(S)$ with $S$ as operator systems, the above definition implies  that if S is a hyperrigid  operator systems then the  identity map  is a noncommutative BKW-operator  for $S$.  Motivated by this, we introduce hyperrigidity with respect to a CP-map. 

\begin{definition}\label{eta-hyperrigidity}
Let $A$ be a $C^*$-algebra in $B(H)$ and $S$ be a set of generators of $A$. Let $\eta: B(H)\rightarrow B(H)$ be a UCP-map. Then 
$S$ is said to be  hyperrigid with respect to $\eta$ if for every sequence of UCP maps 
$\psi_n:B(H)\rightarrow B(H)$, $n=1,2,...,$ and any faithful representation $\sigma: A \to B(H)$,
$$\lim_{n\rightarrow \infty}||\psi_n(\sigma(s))-\eta(\sigma(s))||=0,\forall ~s\in S\Rightarrow \lim_{n\rightarrow \infty}||\psi_n(\sigma(a))-\eta(\sigma(a))||=0, \forall ~a\in A.$$ 
\end{definition} 

As remarked earlier, if $\sigma: A \to B(H)$ is any faithful representation,  then identifying $\sigma(S)$ with $S$ as operator systems, the Definition \ref{eta-hyperrigidity} implies  that 
%if S is a hyperrigid with respect to a UCP-map $\eta$ 
$\eta$  is a noncommutative BKW-operator  for $S$. It is clear that the Definition  \ref{hyperrigidity of Arveson} is a special case of  Definition      \ref{eta-hyperrigidity} if the UCP-map  $\eta$  is the identity map.

\smallskip
Now, we characterize hyperrigidity with respect to a CP-map along the lines of Arveson\cite[Theorem 2.1]{Arv11}.  Although the proof of the following theorem is similar to that of \cite[Theorem 2.1]{Arv11}, with a slight modification, we provide it here for the sake of completeness.

\begin{theorem}\label{characterisation}
Let $S$ be a separable operator system that generate  a $C^*$-algebra $A$. Let $\eta: A\rightarrow A$ be a UCP-map. Then the following statements are equivalent: 
 \begin{enumerate}
     \item[(i)] $S$ is  hyperrigid with respect to $\eta$ in $A$. 
     \item[(ii)] For every non-degenerate representation $\pi: A \rightarrow B(H)$ of $A$ on  a separable Hilbert space $H$  and for any sequence of UCP-maps $\phi_{n}: A \rightarrow B(H)$ such that  $$\lim\limits_{n\rightarrow \infty}\Vert \phi_{n}(s)-\pi(\eta(s))\Vert =0 ~\forall s\in S \implies $$ 
     $$ \lim\limits_{n\rightarrow \infty}\Vert \phi_{n}(a)-\pi(\eta(a))\Vert =0~\forall a\in A.$$
     \item[(iii)] For every non-degenerate representation $\pi: A \rightarrow B(H)$ of $A$ on a separable Hilbert space $H$, the UCP-map $\pi\circ \eta _{|S}$ has unique CP-extension to $A$, namely $\pi \circ \eta $  itself.
    
     \item[(iv)] For any $C^*$-algebra $B$, and for any  unital $*$-homomorphism $\rho: A \rightarrow B$ and UCP-map $\phi: B \rightarrow B$ such that $\phi(\rho(s))=\rho(\eta(s)) ~\forall s \in S$ then $\phi(\rho(a))=\rho(\eta(a)) ~\forall a \in A$.
 \end{enumerate}
\end{theorem}

\begin{proof} 
$(i) \Longrightarrow (ii)$ : Assume $S$ is $\eta$-hyperrigid. Let $\pi: A \rightarrow B(H)$ be a non-degenerate representation of $A$ on  a separable Hilbert space $H$  and  let $\phi_{n}: A \rightarrow B(H)$ be a sequence of UCP-maps such that $$\lim\limits_{n\rightarrow \infty}\Vert \phi_{n}(s)-\pi(\eta(s))\Vert =0 ~\forall s\in S.$$ Consider a faithful representation $\sigma: A \rightarrow B(K)$, the direct sum representation $\sigma \oplus \pi: A \rightarrow B(H\oplus K)$ is also a faithful representation of $A$. 
Consider the UCP-maps $\sigma_{n}: (\sigma \oplus \pi)(A) \rightarrow B(H\oplus K)$ given by 
$$\sigma_{n}(\sigma \oplus \pi)(a)=\sigma(\eta(a)) \oplus \phi_{n}(a).$$
Then by  Arveson extension theorem \cite{Arv69}  we have a UCP-map $\tilde{\sigma_{n}}:  B(H\oplus K)\rightarrow  B(H\oplus K)$ that extends $\sigma_{n}$. 
Now, for each $s\in S$,
\begin{align*}
\Vert \tilde{\sigma_{n}}(\sigma \oplus \pi)(s)-(\sigma(\eta) \oplus \pi)(s) \Vert =& \Vert \sigma_{n}(\sigma \oplus \pi)(s) - (\sigma(\eta) \oplus \pi)(s) \Vert \\
=& \Vert \sigma(\eta(s)) \oplus \phi_{n}(s) - \sigma((\eta(s)) \oplus \pi(s) \Vert \\
=& \Vert \phi_{n}(s) - \pi(s)  \Vert \rightarrow 0 ~as ~ n \rightarrow \infty.
\end{align*}
Thus the sequence of UCP-maps $\tilde{\sigma_{n}}$ converges to identity map of $(\sigma(\eta) \oplus \pi)(A)$ on the operator system $(\sigma \oplus \pi)(S)$

Since $\phi_{m}(s)$ converges to $\pi(s)$ on $H_{n}$ for each $s\in S$, we have that $\tilde{\sigma_{n}}(s)$ converges $(\sigma \oplus \pi )(s)$ on $H_{n}\oplus H_{n}'$. By assumption $(i)$, it follows that $\tilde{\sigma_{n}}(a)$ converges $(\sigma \oplus \pi )(a)$ on $H_{n}\oplus H_{n}'$ for all $a\in A$. Then $\forall ~ a\in A$,
 \begin{align*}
 \limsup\limits_{n\rightarrow \infty} \Vert \phi_{n}(a)-\pi(a) \Vert_{n} &\leq  \limsup\limits_{n\rightarrow \infty} \Vert  \sigma(a)\oplus \phi_{n}(a)-\sigma(a)\oplus\pi(a) \Vert_{n}\\
 &= \limsup\limits_{n\rightarrow \infty} \Vert  \sigma_{n}(\sigma(a)\oplus \phi_{n}(a))-\sigma(a)\oplus\pi(a) \Vert_{n}\\
 &=0.
 \end{align*}
 
$(ii) \Longrightarrow (iii)$: Follows easily.
\end{proof}

\begin{corollary}
Let $S$ be an operator system in a $C^*$-algebra $A\subseteq B(H)$ such that $A=C^*(S)$. If $S$ is  hyperrigid, then  $S$ is   hyperrigid with respect to $\eta$ for each surjective UCP-map  $\eta : A\rightarrow A$.  
\end{corollary} 

%%%%%%%%%%%%%%%%%%%%%%%%%%%%%%%%%%%%%%
From statements (i) and (iii) of Theorem~\ref{characterisation}, together with Theorem~\ref{BKW-operators and unique CP-extensions}, it follows that for $A = B(H)$, an operator system $S$ is hyperrigid with respect to $\eta$ if and only if $\pi \circ \eta$ is a noncommutative BKW-operator for $S$ for every non-degenerate representation $\pi$.
%%%%%%%%%%%%%%%%%%%%%%%%%%%%%%%%%%%%

Similar to  \cite[Problem $5.3$ and Problem $5.4$]{Alt10}, the Definition \ref{eta-hyperrigidity} inherently leads to the following two natural problems:

\begin{problem}\label{Problem 1}
Given a $C^*$-algebra $A$, a Hilbert space $H$ and a completely positive map $ \eta :A\rightarrow B(H) $, find conditions under which there exists a nontrivial  operator system $S$ which is hyperrigid for $\eta$ in $A$. In this case, try to determine some or all of them.
\end{problem} 

\begin{problem}\label{Problem 2}
Given a  $C^*$-algebra $A$, a Hilbert space $H$ and an operator system $S$ of $A$, try to determine some or all of the completely positive linear maps $\eta :A\rightarrow B(H)$  for which $S$ is a hyperrigid set for $\eta$. 
\end{problem} 

Theorem \ref{existence of BKW-operator} establishes the existence of a noncommutative BKW-operator for certain Hilbert spaces. In fact, boundary representations of an operator system $S$ form a class of noncommutative BKW-operators for $S$. Exploration of Problems \ref{Problem 1} and Problem \ref{Problem 2} is left for future research.

In view of \Cref{BKW-operators and unique CP-extensions}, Problem \ref{Problem 2} may be restated as follows.

\begin{problem}
What are the completely positive maps on an operator system that admit a unique completely positive extension to the generated $C^*$-algebra?
\end{problem} 
To the best of our knowledge, CP-maps with unique CP-extensions have not been addressed in the existing literature, except in the case of representations.

\section{A noncommutative Korovkin type Theorem}

In the previous sections, we considered UCP-maps. In this section, we turn our attention to nonunital CP-maps in order to present a noncommutative analogue of the operator version of Korovkin’s theorem, recently established by Popa\cite[Theorem 1]{Pop22}.  Let us first formally state the result as follows.\\

Let $\textbf{1}, e_{1}\text{ and } e_{2}$ denotes the functions $\textbf{1}(x)=1,e_{1}(x)=x\text{ and }e_{2}(x)=x^2$ on $[a,b]$, respectively.
 \begin{theorem}\cite[Theorem 1]{Pop22}
Let $X$ be a compact Hausdorff space, $V_{n}:C[a,b]\rightarrow C(X)$ be a sequence of  linear positive operators and $V:C[a,b]\rightarrow C(X)$ a linear positive operator such that $V(\textbf{1})V(e_{1})=[V(e_{1})]^2$ and $V(e_{1})(x)>0$, $\forall~x\in X$. If $\lim\limits_{n\rightarrow \infty}V_{n}(\textbf{1})=V(\textbf{1})$, $\lim\limits_{n\rightarrow \infty}V_{n}(e_{1})=V(e_{1})$ and $\lim\limits_{n\rightarrow \infty}V_{n}(e_{2})=V(e_{2})$  all uniformly on $X$, then for every $f\in C[a,b]$,  $\lim\limits_{n\rightarrow \infty}V_{n}(f)=V(f)$ uniformly on $X$.
 \end{theorem}

Here, we provide a noncommutative version of Popa’s result, and the proof is obtained by adapting the technique used in \cite{AP25}. While searching for a noncommutative analogue of a classical result, one typically replaces positivity with complete positivity. However, in this case, the result holds even for more general maps—namely, Schwarz maps and 2-positive maps. Here we first establish the result in a more general setting, from which the case of completely positive (CP) maps is derived as a direct consequence, see \Cref{noncommutative  theorem}.

Let $A$ and B be  unital \( C^* \)-algebras. Recall that a linear map $\phi: A \to B$ is said to be a Schwarz map if $\phi(x^*x)\geq \phi(x)^*\phi(x)$ for all $x \in A$.  Now fix an element \( a \in A \), and denote by  \( A_0 = C^*(a) \), the \( C^* \)-subalgebra of \( A \) generated by \( a \). Then we have,

\begin{theorem}\label{noncommutative Popa's theorem}
Let  $\phi_{n}: A_{0}\rightarrow B$ be a sequence of Schwarz maps and $\phi: A\rightarrow B$ be a 2-positive linear map  such that $\phi(1)$ is invertible,   $\phi(a^*a)=\phi(a^*)\phi(1)^{-1}\phi(a)$ and  $\phi(aa^*)=\phi(a)\phi(1)^{-1}\phi(a^*)$. If $\lim\limits_{n\rightarrow \infty} \phi_{n}(1)=\phi(1)$,  $\lim\limits_{n\rightarrow \infty} \phi_{n}(a)=\phi(a)$,\\
$\lim\limits_{n\rightarrow \infty} \phi_{n}(a^*a)=\phi(a^*a)$, $\lim\limits_{n\rightarrow \infty} \phi_{n}(aa^*)=\phi(aa^*)$ in norm, then for every $x\in C^*(a)$, $\lim\limits_{n\rightarrow \infty} \phi_{n}(x)=\phi(x)$ in norm. 
\end{theorem}
\begin{proof}
Consider the linear map $\psi:A_{0}\rightarrow B$ defined by $$\psi(x)=\phi(1)^{-\frac{1}{2}}\phi(x)\phi(1)^{-\frac{1}{2}}~~\forall x\in A_{0}.$$
Then $B$ is a positive map as $\phi(1)$ is self adjoint and conjugation preserves positivity. Also,
$$\psi(1)=\phi(1)^{-\frac{1}{2}}\phi(1)\phi(1)^{-\frac{1}{2}}=1,$$
$$\psi(a)=\phi(1)^{-\frac{1}{2}}\phi(a)\phi(1)^{-\frac{1}{2}} \text{ and }$$
\begin{align*}
   \psi(a^*a)&=\phi(1)^{-\frac{1}{2}}\phi(a^*a)\phi(1)^{-\frac{1}{2}} \\
            &=\phi(1)^{-\frac{1}{2}} \phi(a^*)\phi(1)^{-1}\phi(a)\phi(1)^{-\frac{1}{2}} \\
            &= \psi(a^*)\psi(a).
\end{align*}
$$\text{Similarly,  }~\psi(aa^*)=\psi(a)\psi(a^*).$$

By \cite[Theorem 3.18 (iii)]{Pau02},  $\psi$ is a $*$-homomorphism on $A_{0}$. 

Now, note that, a new  sequence of 2-positive linear maps  
$$\psi_{n}=\phi(1)^{-\frac{1}{2}}\phi_{n}\phi(1)^{-\frac{1}{2}}$$
has the following properties: 
$$\psi_{n}(1)=\phi(1)^{-\frac{1}{2}}\phi_{n}(1)\phi(1)^{-\frac{1}{2}}\rightarrow 1 =\psi(1)\text{ as } n\rightarrow \infty ,$$ 
\begin{align*}
    \Vert \psi_{n}(a)-\psi(a)\Vert &= \Vert \phi(1)^{-\frac{1}{2}}\phi_{n}(a)\phi(1)^{-\frac{1}{2}}-\phi(1)^{-\frac{1}{2}}\phi(a)\phi(1)^{-\frac{1}{2}}\Vert \\
                             & \leq \Vert \phi(1)^{-\frac{1}{2}} \Vert ^{2} \Vert \phi_{n}(a)-\phi(a)\Vert 
                              \rightarrow 0 \text{ as } n\rightarrow \infty
\end{align*}
\begin{align*}
     \Vert \psi_{n}(a^*a)-\psi(a^*a)\Vert &= \Vert \phi(1)^{-\frac{1}{2}}\phi_{n}(a^*a)\phi(1)^{-\frac{1}{2}}-\phi(1)^{-\frac{1}{2}}\phi(a^*a)\phi(1)^{-\frac{1}{2}}\Vert \\
                             & \leq \Vert \phi(1)^{-\frac{1}{2}} \Vert ^{2} \Vert \phi_{n}(a^*a)-\phi(a^*a)\Vert 
                              \rightarrow 0 \text{ as } n\rightarrow \infty.
\end{align*}
Similarly $ \Vert \psi_{n}(a^*a)-\psi(a^*a)\Vert \rightarrow 0 \text{ as } n\rightarrow \infty $.
Thus $\psi_{n}(1)\rightarrow \psi(1)$, $ \psi_{n}(a)\rightarrow \psi(a) $, $ \psi_{n}(a^*a)\rightarrow \psi(a^*a) $ and $ \psi_{n}(aa^*)\rightarrow \psi(aa^*) $  as $ n\rightarrow \infty $.

Then, since $\psi$ is a $*$-homomorphism   and by \cite[Corollary 1.4]{LN84}, we have 
$$ \psi_{n}(x)\rightarrow \psi(x)   \text{ as }  n\rightarrow \infty ~~\forall x\in A_{0}).$$ 

Now, $\forall  x\in A_{0}$, 

\begin{align*}
\Vert \phi_{n}(x)-\phi(x) \Vert &=\Vert \phi(1)^{\frac{1}{2}}\psi_{n}(x)\phi(1)^{\frac{1}{2}}- \phi(1)^{\frac{1}{2}}\psi(x)\phi(1)^{\frac{1}{2}} \Vert \\
                          &\leq \Vert \phi(1)^{\frac{1}{2}} \Vert ^{2} \Vert \Vert \psi_{n}(x)- \psi(x) \Vert 
                         \rightarrow 0 \text{ as } n\rightarrow \infty .
\end{align*}

\end{proof} 

\begin{remark}
    The assumption of $\phi_{n}$ to be a Schwarz map is to use the result   \cite[Corollary 1.4]{LN84} and  the assumption $\phi$ to be a 2-positive is to use \cite[Theorem 3.18 (iii)]{Pau02}.  Note that, in \cite[Theorem 3.18 (iii)]{Pau02}) $\phi$ is assumed to be CP but the proof requires only 2-positivity of $\phi$. 
\end{remark} 

Since CP-maps are  both Schwarz and 2-positive, we obtain the following corollary to Theorem \ref{noncommutative Popa's theorem}.

\begin{corollary}\label{noncommutative  theorem}
Let  $\phi,\phi_{n}: A_{0}\rightarrow B$ be a sequence of CP-maps  such that $\phi(1)$ is invertible,   $\phi(a^*a)=\phi(a^*)\phi(1)^{-1}\phi(a)$ and  $\phi(aa^*)=\phi(a)\phi(1)^{-1}\phi(a^*)$. If $\lim\limits_{n\rightarrow \infty} \phi_{n}(1)=\phi(1)$,  $\lim\limits_{n\rightarrow \infty} \phi_{n}(a)=\phi(a)$,
$\lim\limits_{n\rightarrow \infty} \phi_{n}(a^*a)=\phi(a^*a)$, $\lim\limits_{n\rightarrow \infty} \phi_{n}(aa^*)=\phi(aa^*)$ in norm, then for every $x\in C^*(a)$, $\lim\limits_{n\rightarrow \infty} \phi_{n}(x)=\phi(x)$ in norm. 
\end{corollary}

\begin{remark}
Corollary \ref{noncommutative theorem} is a noncommutative version of Popa's theorem \cite[Theorem 1]{Pop22}, as the condition $\phi(a^*a)=\phi(a^*)\phi(1)^{-1}\phi(a)$ in Corollary \ref{noncommutative theorem} is equivalent to the condition $V(\textbf{1})V(x^2)=V(x^2)$ in  \cite[Theorem 1]{Pop22} for the commutative case, with $V=\phi$. 
\end{remark}
By removing the unitality assumption on the  CP-maps 
$\phi$ and $\phi_{n}$ in the definition of a noncommutative BKW-operator and thereby defining a non-unital noncommutative BKW-operator, then Corollary \ref{noncommutative theorem} implies that any map $\phi$ satisfying the conditions of Corollary \ref{noncommutative theorem} qualifies as a non-unital noncommutative BKW-operator for the operator system generated by $\{1,a,aa^*, a^*a\}$.

We present a class of maps that satisfy the conditions of the Theorem \ref{noncommutative Popa's theorem}.

\begin{example}
Let $H$ be a Hilbert space and $T$  be an invertible operator on $H$. Let $\lambda > 0$ be a real number. Define  $\phi: C^*(T)\rightarrow B(H)$ be the map $$\phi(x)=\lambda T^*xT.$$ Then $\phi$ is a CP map, and thus $2$-positive. Then $$\phi(I)=\lambda T^*T,\text{ which is invertible}.$$
Note that if $T$ is unitary and  $\lambda=1$  then 
$\phi$ is a representation; otherwise, it need not be multiplicative.
Also,  $\phi(T^*)=\lambda T^*T^*T=\lambda (T^*)^2T$ and $\phi(T)=\lambda T^*TT= \lambda T^*T^2$. Now,
\begin{align*}
    \phi(T^*T)&= \lambda T^*(T^*T)T = \lambda T^* (T^*T^2) 
    =  T^*\phi(T) \\
              &=  T^* \phi(I) \phi(I)^{-1}\phi(T)
              =  T^* (\lambda T^*T)\phi(I)^{-1}\phi(T) \\
              &=  \phi(T^*)\phi(I)^{-1}\phi(T)\\
    \phi(TT^*)&= \lambda T^*TT^*T 
              = \lambda  T^*T(TT^{-1})((T^*)^{-1}T^*)T^*T \\ 
              & = \lambda T^*T^2 (T^{-1}(T^*)^{-1})(T^*)^2T 
              =  \lambda T^*T^2 (\frac{1}{\lambda} (T^*T)^{-1}) \lambda (T^*)^2T \\
              &= \phi(T)\phi(I)^{-1}\phi(T^*)
\end{align*}
That is $\phi$ satisfies the conditions of Theorem \ref{noncommutative Popa's theorem}.
\end{example} 

%%%%%%%%%%%%%%%%%%%%%%%%%%%%%%%%%%%%%%%%%%%%%%%%%
\section{Concluding Remarks and Future Directions} 
In this paper, we have initiated a study of Noncommutative BKW-operators in the setting of $C^*$-algebras and CP-maps. We have identified a connection between noncommutative BKW-operators with uniqueness in  the Arveson extension theorem for CP-maps.
Also, the existence of noncommutative BKW-operator  for any given operator system is addressed. Additionally, we have discussed several examples and explored different notions of noncommutative BKW-operators and their interconnections. Finally, we established a noncommutative analogue of a recently proposed Korovkin-type theorem. 

This work highlights a number of natural questions that remain open for future investigation. In particular, extending the main results of this article to the nonseparable setting requires further study. Furthermore, the uniqueness aspect of the Arveson extension theorem has not yet been addressed either in the existing literature or in the present work.

Potential directions for future research include:
\begin{enumerate}
    \item A natural direction is to examine to what extent Theorem \ref{BKW-operators and unique CP-extensions} and subsequent results can be extended beyond separable Hilbert spaces to the more general non-separable framework.
    \item A promising direction for future research is a deeper investigation of the uniqueness aspect of the Arveson extension  theorem, which to date has received only limited attention in the literature and in the present article.
    \item Exploring further connections between noncommutative BKW-operators and hyperrigidity with respect to UCP-maps. 
    \item Given a $C^*$-algebra $A$, a Hilbert space $H$ and a  CP-map $ \eta :A\rightarrow B(H) $, find conditions under which there exists a nontrivial  operator system $S$ which is hyperrigid for $\eta$ in $A$. In this case, try to determine some or all of them.   
    \item Given a  $C^*$-algebra $A$, a Hilbert space $H$ and a generating operator system $S$ of $A$, try to determine some or all of the 
    CP-maps $\eta :A\rightarrow B(H)$  for which $S$ is a hyperrigid set for $\eta$. 
\end{enumerate}

%%%%%%%%%%%%%%%%%%%%%%%%%%%%%%%%%%%%%%%%%%%%%%%%%%%%%

{\bf Acknowledgments} The authors would like to thank Prof. B.V. Rajarama Bhat, ISI Bangalore, for the valuable discussions related to Section 3. 

\textbf{Funding :} No funding was received for this research.

\textbf{Data Availability:}
This paper has no associated data. 

\textbf{ Declarations}\\
\textbf{Conflicts of interests}
The authors have no relevant financial or non-financial interests to disclose.

\bibliographystyle{amsplain}

\end{document}